\documentclass[10pt]{article}
\usepackage{mathrsfs}
\usepackage{amsfonts}
\usepackage{titlesec}
\usepackage{enumerate}
\usepackage[toc,page,title,titletoc,header]{appendix}
\usepackage{amsthm,amsmath,amssymb,anysize}
\usepackage[numbers,sort&compress]{natbib}
\theoremstyle{definition}
\newtheorem{lemma}{Lemma}
\newtheorem{theorem}[lemma]{Theorem}
\newtheorem{remark}[lemma]{Remark}

\newtheorem{corollary}[lemma]{Corollary}

\newcommand{\periodafter}[1]{#1.}
\titleformat{\subsection}[runin]
{\normalfont\bfseries}{\thesubsection}{0.5em}{\periodafter}
\marginsize{4.2cm}{4.2cm}{3.6cm}{3cm}
\numberwithin{equation}{section}

\renewcommand{\thesubsection}{\arabic{section}.\arabic{subsection}.}

\newcounter{RomanNumber}

\newcommand{\MyRoman}[1]{\setcounter{RomanNumber}{#1}\Roman{RomanNumber}}
\title{On the maximal abelian  subalgebras of the general linear Lie color algebras}
\author{\textsc{Shujuan Wang$^{1}$ and
  \textsc{Wende Liu$^{2,}$}}\footnote{Correspondence:  wendeliu@ustc.edu.cn (W. Liu)}\\
  {\small \textit{$^1$Department of Mathematics, Shanghai Maritime University,}}\\
\small \textit{Shanghai 201306, China}  \\
 \small\textit{$^2$School of Mathematics and Statistics,}
\textit{Hainan Normal University,}\\ \small\textit{ Haikou 571158,  China} }
\date{ }

\begin{document}
\maketitle
\begin{quotation}
\small\noindent \textbf{Abstract}:
Let $\Gamma$ be a finite group and $V$
 a finite-dimensional $\Gamma$-graded space over an  algebraically closed field of characteristic not equal to 2.
In the sense of conjugation, we classify all the so-called pre-nil  or nil  maximal abelian subalgebras
for the general linear Lie color algebra $\frak{gl}(V,\Gamma)$.
In the situation of $\Gamma$ being a  cyclic group, we determine the minimal dimensions of pre-nil  or nil faithful representations
for any finite-dimensional abelian Lie color algebra.

\vspace{0.2cm} \noindent{\textbf{Keywords}}:  faithful representations; abelian subalgebras; Lie color algebras

\vspace{0.2cm} \noindent{\textbf{Mathematics Subject Classification 2010}}: 17B10, 17B30, 17B50, 17B75
\end{quotation}

\setcounter{section}{0}
\section{Introduction}
Lie color algebras, which were introduced by Ree in 1960 (see \cite{Ree}),
play an important role in mathematics and theoretical physics,
especially in the conformal field theory and supersymmetries (see \cite{Ano,Wei}).
For example, Scheunert and Zhang studied cohomology of Lie color algebras (see \cite{schunertzhang});
Su, Zhao and Zhu constructed simple Lie color algebras of generalized Witt type and Weyl type (see \cite{su1} and \cite{su2}).

Let $L$ be a finite-dimensional Lie (color) algebra and write
$$
\mu(L)=\min \{\dim V\mid \mbox{$(\rho, V)$ is a faithful representation of $L$}\}.
$$
Ado's theorem of Lie (color) version guarantees the existence of $\mu(L)$ (see \cite[p.719]{Scheunert}).
In 1905, I. Schur  determined the maximal dimension of abelian subalgebras
for the general linear Lie algebra over $\mathbb{C}$ in \cite{Schur} and then $\mu(L)$ can be determined for any finite-dimensional abelian Lie algebra $L$ (see also \cite{Jacobson,Mirzakhani}).
A  super-version of  Schur's work  was given for Lie superalgebras over $\mathbb{C}$ (see \cite{liu-wangabel,liu-wangsubalgebra}).
In this paper, we shall  offer a color-version of Schur's work  for Lie color algebras over an algebraically closed field of characteristic not equal to 2.
We also determine the minimal pre-nil  or nil faithful representations for any finite-dimensional abelian Lie color algebra with respect to a finite cyclic group.

As is well-known, the function $\mu$  plays an important role on affine crystallographic groups
and finitely generated torsion-free nilpotent groups.
Benoist,
Burde and Grunewald (see \cite{Benoist,Burde2})  gave an example of a nilpotent Lie
algebra $L$  such that $\mu(L)>\dim L+1$, which is a counterexample of
Milnor's conjecture (see \cite{Milnor}).
However, it is  not easy to compute $\mu(L)$, even the bounds of $\mu(L)$ for a given finite-dimensional Lie (color) algebra.
A lot of  work on the function $\mu$ revolves around nilpotent or (semi)simple Lie (super)algebras as well as their various  extensions (see \cite{Schur,Jacobson,Mirzakhani,Burde3,DB,Burde4,CR,CR1,Birkhoff,
Reed,Benoist,liu-wang,liu-wangabel,wang-chen-liu,liu-chen,liu-wangsubalgebra,liu-wangcur} for example).

Throughout $\mathbb{F}$ is an algebraically closed field of characteristic not equal to 2
and all vector spaces and algebras are over $\mathbb{F}$  and of finite
dimensions.

\subsection{Generalized matrix units}
 Denote by $\mathrm{M}(m)$  the set consisting of  all $m\times m$ matrices over $\mathbb{F}$.
We define a total order on the set $\{(i, j)\mid 1\leq i,j\leq m\}$ by
 $$(i,j)<(k, l)\Longleftrightarrow i<k, \quad\mbox{or}\quad i=k \quad\mbox{but}\quad j< l.$$
Denote by $\mathrm{ht}(X)=\min\{(i, j)\mid x_{ij}\neq 0\}$
for $X=(x_{ij})\in \mathrm{M}(m)$ and $\mathrm{ht}(0)=(\infty, \infty)$.
Hereafter $a_{mn}$ is  $(m,n)$-entry of the matrix $A=(a_{ij})$.
An  element $X$  is called a generalized matrix unit of $(i, j)$-form, usually written as $u_{i,j}$,
if $\mathrm{ht}(X)=(i, j)$
and its $i$th row is $(0,...,0,1,0...,0)$,
where $1$ is at the $j$-th position.
 In general, for given $(i, j)$, there are many generalized
matrix units $u_{i,j}$.
For a subspace $S$ of $\mathrm{M}(m)$, write
$$\mathrm{ht}(S):= \min\{\mathrm{ht}(X)\mid X\in S\}.$$
Denote by $\mathrm{I}_{m}=\Sigma^{m}_{i=1}e_{ii}$ the  $m\times m$ identity matrix.
Hereafter $e_{ij}$ is the matrix unit having 1 at $(i, j)$-position and 0 elsewhere.
Fix a nonzero element $a\in \mathbb{F}$ and let $\mathrm{T}_{ij}(a)$ or $\mathrm{D}_{i}(a)$
 be  $\mathrm{I}_{m}+ae_{ij}$ or  $\mathrm{I}_{m}+(a-1)e_{ii}$, respectively.
As in \cite{liu-wangsubalgebra},
we introduce the following similar operations of $\mathrm{M}(m)$:
\begin{enumerate}
\item[(1)]
\textit{$t$-type operators}
$$\mbox{$s_{\mathrm{T}_{ij}(a)}:=\mathrm{l}_{\mathrm{T}_{ij}(a)^{-1}}\mathrm{r}_{\mathrm{T}_{ij}(a)}$ for $1\leq i<j\leq m$ and $a\in \mathbb{F}$}.$$
\item[(2)]
\textit{$d$-type operators}
$$\mbox{$s_{\mathrm{D}_{i}(a)}:=\mathrm{l}_{\mathrm{D}_{i}(a)^{-1}}\mathrm{r}_{\mathrm{D}_{i}(a)}$ for $1\leq i\leq m$ and $a\in \mathbb{F}\backslash\{0\}$},$$
\end{enumerate}
where, $\mathrm{l}_{A}$
or $\mathrm{r}_{A}$ is
the operator of left or right associative multiplication by the matrix $A$ for $A\in \mathrm{M}(m)$, respectively.
If $A=(a_{ij})$ and $\mathrm{ht}(A)=(i, j)$, we write
$$\mathrm{h}_{A}=\sum_{l> j} s_{\mathrm{T}_{jl}(-a_{il})}s_{\mathrm{D}_{i}(a_{ij})}.$$

\subsection{Lie color algebras}
Let $\Gamma$ be an  abelian group and $V=\oplus_{\alpha\in \Gamma} V_{\alpha}$ a  $\Gamma$-graded space.
The elements in $\cup _{\alpha\in\Gamma}V_{\alpha}$ are said to be
 homogenous. For a  homogeneous  element $v\in V_{\alpha}, \alpha\in \Gamma$,
we set $|v|=\alpha,$ the degree of $v.$
In addition, the symbol $|x|$ implies
that $x$ is homogeneous.
By definition, a $\Gamma$-graded algebra $\frak{g}$ is
a $\Gamma$-graded space  $\frak{g}=\oplus_{\alpha\in \Gamma}\frak{g}_{\alpha}$
with a bilinear  multiplication consistent with the $\Gamma$-gradation.
 Let $\frak{g}$ be a $\Gamma$-graded algebra.
If the multiplication of $\frak{g}$ is trivial, we say $\frak{g}$ to be abelian;
if the multiplication of $\frak{g}$ satisfies the associativity law, we say $\frak{g}$ to be associative.

Let $\Gamma$ be an  abelian group.
A bi-character on $\Gamma$ is a map
$\varepsilon: \Gamma\times \Gamma\longrightarrow \mathbb{F}\backslash\{0\}$ such that
$$\varepsilon(\alpha, \beta)\varepsilon(\beta, \alpha)=1, \quad\varepsilon(\alpha\beta, \gamma)=\varepsilon(\alpha, \gamma)\varepsilon(\beta, \gamma),$$
where $\alpha, \beta, \gamma\in \Gamma$.
Let $\Gamma$ be an  abelian group of a bi-character $\varepsilon$
and $\frak{g}$ a $\Gamma$-graded algebra, whose multiplication is denoted by $[\,,\,]$.
If
$$[x, y]+\varepsilon(|x|, |y|)[y, x]=0,$$
$$\varepsilon(|z|, |x|) [x, [y, z]] + \varepsilon(|x|, |y|)[y, [z, x]] + \varepsilon(|y|, |z|)[z, [x, y]] = 0,$$
where $x, y, z\in \frak{g}$, then $\frak{g}$ is called a Lie color algebra.
 Lie (super)algebras are  Lie color algebras (see \cite[p.525]{su2}).
If $\frak{A}=\oplus_{\alpha\in \Gamma}\frak{A}_{\alpha}$ is a $\Gamma$-graded associative algebra,
by introducing a new multiplication
$$[x, y] = xy -\varepsilon(|x|, |y|)yx,\quad x, y\in \frak{A},$$
$\frak{A}$ becomes a Lie color algebra, which is denoted by $\frak{A}^-$.

\subsection{The general linear Lie color algebras}\label{2004121756}
Let $\Gamma$ be an abelian group
 of a bi-character $\varepsilon$ and  $V=\oplus_{\alpha\in \Gamma} V_{\alpha}$ a  $\Gamma$-graded space.
Then $\mathrm{End}_{\mathbb{F}}(V)$ is a  $\Gamma$-graded associative algebra
 and   $\mathrm{End}_{\mathbb{F}}(V)^-$ is a Lie color algebra in the usual fashion,
 which is  called  the general linear Lie color algebra and denoted by   $\frak{gl}(V, \Gamma)$  (see \cite[p.447]{zhang2013}).
Let $\frak{a}$ be a subalgebra of $\frak{gl}(V,\Gamma)$.
If each element in $\frak{a}$ is nilpotent,
we say $\frak{a}$ to be nil;
if each element in $\cup_{\alpha\in \Gamma^*}\frak{a}_{\alpha}$ is nilpotent,
we say $\frak{a}$ to be pre-nil.
Hereafter $\Gamma^*$ denotes the set $\Gamma\backslash\{0\}$, where $0$ is the identity of $\Gamma$.
If $\frak{a}$ is a (pre-nil  or nil) abelian subalgebra of the maximal dimension for $\frak{gl}(V,\Gamma)$,
we say $\frak{a}$ to be (pre-nil  or nil) maximal abelian.
Let $L=\oplus_{\alpha\in \Gamma} L_{\alpha}$ be a Lie color algebra
and $\rho: L\longrightarrow\frak{gl}(V,\Gamma)$ a linear map of degree $0$.
If $\rho([x, y])=[\rho(x), \rho(y)]$ for $x, y\in L$,
we call $(\rho, V)$ to be a representation of $L$.
Let $(\rho, V)$ be a representation of $L$. If $\ker\rho=0$, we say $(\rho, V)$ to be faithful;
if each element in $\rho(V)$ is nilpotent, we say $(\rho, V)$ to be nil;
if each element in $\cup_{\alpha\in \Gamma^*}\rho(V)_{\alpha}$ is nilpotent,
we say $(\rho, V)$ to be pre-nil.
Write
\begin{align*}
\mu_{\text{nil}}(L)&=\min \{\dim V\mid \mbox{$(\rho, V)$ is a nil faithful representation of $L$}\},\\
\mu_{\text{pre-nil}}(L)&=\min \{\dim V\mid \mbox{$(\rho, V)$ is a pre-nil  faithful representation of $L$}\}.
\end{align*}

\subsection{The matrix-version of $\frak{gl}(V,\Gamma)$}
Hereafter, we make a convention that the symbol $\Gamma$ always denotes
a finite abelian group of a bi-character $\varepsilon$,
whose all elements are $\alpha_1, \ldots, \alpha_k$.
For the fixed order of $\alpha_1, \ldots, \alpha_k$,
if $V=\oplus^k_{i=1} V_{\alpha_i}$ is a  $\Gamma$-graded space
and $\dim V_{\alpha_i}=m_i$, we say that $(m_1,\ldots, m_k)$ is $\Gamma$-dimension of $V$.
Let $V$ be of $\Gamma$-dimension $(m_1, \ldots, m_k)$ and $m=\sum^k_{i=1}m_i$.
A rearrangement  of
$$(\stackrel{m_{1}}{\overbrace{\alpha_{1}, \ldots, \alpha_{1}}},
\stackrel{m_{2}}{\overbrace{\alpha_{2}, \ldots, \alpha_{2}}}, \ldots,
\stackrel{m_{k}}{\overbrace{\alpha_{k}, \ldots, \alpha_{k}}})$$
is called an $(m_1, \ldots, m_k)$-tuple.
Denote by $B=(v_{1},\ldots,v_{m})$  an ordered homogeneous basis of $V$
and equip it with an $(m_1, \ldots, m_k)$-tuple
$\Phi =(|v_{1}|,\ldots,|v_{m}|)$ induced by $B$.
Then the general linear Lie color algebra $\frak{gl}(V, \Gamma)$ is isomorphic to its  matrix-version
$\frak{gl}^{\Phi }(m_1,\ldots,m_k)$ induced by tuple $\Phi $, which
has the underlying matrices space $\mathrm{M}(m)$ and  a $\Gamma$-grading structure
$$\frak{gl}^{\Phi }(m_1,\ldots,m_k)=\oplus^k_{l=1}\frak{gl}^{\Phi }_{\alpha_{l}}(m_1,\ldots,m_k),$$
 where
$$\mbox{$\frak{gl}^{\Phi }_{\alpha_{l}}(m_1,\ldots,m_k)=\mathrm{Span}\{e_{ij}\mid
\Phi _i-\Phi _j=\alpha_{l}\}$}.$$
Hereafter, $\Phi_n$ denotes the $n$-th entry of the tuple $\Phi$.
In particular, if
$$\Phi=(\stackrel{m_{1}}{\overbrace{\alpha_{1}, \ldots, \alpha_{1}}},
\stackrel{m_{2}}{\overbrace{\alpha_{2}, \ldots, \alpha_{2}}}, \ldots,
\stackrel{m_{k}}{\overbrace{\alpha_{k}, \ldots, \alpha_{k}}}),$$
$\frak{gl}^{\Phi }(m_1,\ldots,m_k)$ is denoted by $\frak{gl}(m_1,\ldots,m_k)$ for short.
For any $(m_1, \ldots, m_k)$-tuple $\Phi $, write $\frak{t}^{\Phi }(m_1,\ldots,m_k)$ or $\frak{s}^{\Phi }(m_1,\ldots,m_k)$
 for the subalgebras of $\frak{gl}^{\Phi }(m_1,\ldots,m_k)$
consisting of upper triangular or strictly upper triangular matrices, respectively.

\section{Maximal pre-nil or nil abelian subalgebras of $\frak{gl}(V,\Gamma)$}
In this section, we shall give a lower bound of the maximal dimensions
for abelian subalgebras of the general linear Lie color algebra $\frak{gl}(V,\Gamma)$.
In the case $\Gamma=\mathbb{Z}_k$, the cyclic group of order $k$, we shall classify the pre-nil  or nil maximal
abelian subalgebras of  $\frak{gl}(V,\Gamma)$ in the sense of conjugation.
In addition, we shall give a method to determine $\mu_{\text{pre-nil}}(L)$ and $\mu_{\text{nil}}(L)$
for any abelian Lie color algebra $L$.
\subsection{Main lemmas}
The following lemma realizes the triangulation of matrices
for any pre-nil abelian subalgebra of $\frak{gl}(m_1,\ldots, m_k)$,
which is a color-version of \cite[Lemma  3.3]{Y2}.

\begin{lemma}\label{1219}
Let $\frak{a}$ be  a pre-nil  abelian subalgebra of  $\frak{gl}(m_1,\ldots, m_k)$.
Then there exists  an $(m_1, \ldots, m_k)$-tuple $\Phi$
such that $\frak{a}$ is contained in $\frak{t}^{\Phi }(m_1, \ldots, m_k)$.
Furthermore, if $\frak{a}$ is nil abelian, then there exists  an $(m_1, \ldots, m_k)$-tuple $\Phi $
such that $\frak{a}$ is contained in $\frak{s}^{\Phi }(m_1, \ldots, m_k)$.
\end{lemma}
\begin{proof}
From Jacobson's Theorem on weakly closed sets,
it is sufficient to prove that $\frak{a}$ has common homogeneous eigenvectors in $V$.
Since $\frak{a}$ is pre-nil ,
for $A\in \frak{gl}_{\gamma_{l}}(m_1, \ldots, m_k)$ with $\gamma_{l}\in \Gamma^*$,
 zero is the only eigenvalue of $A$, and hence
$$V_1:=\left\{x\in V\mid \frak{gl}_{\gamma_{l}}(m_1, \ldots, m_k)x=0, \gamma_{l}\in \Gamma^*\right\}$$
is nonzero.
It is clear that $V_1$ is $\frak{gl}_{0}(m_1, \ldots, m_k)$-module.
Then
$$V_2:=\left\{x\in V_1\mid \frak{gl}_{0}(m_1, \ldots, m_k)x\in \mathbb{F}x\right\}$$
 is also nonzero. Any nonzero homogeneous element in $V_2$ is a common eigenvector for $\frak{a}$.
\end{proof}

The following lemma gives an HGMU decomposition (see below for a definition)
of any abelian subalgebra for $\frak{s}^{\Phi }(m_1,\ldots,m_k)$,
which is a color-version of \cite[Lemma 2.3 ]{liu-wangsubalgebra}.

\begin{lemma}\label{1527}
 Let $\Phi $ be an $(m_1,\ldots,m_k)$-tuple and $m=\sum^k_{i=1}m_i$.

(1) If  $A$ is a homogeneous element in $\frak{s}^{\Phi }(m_1,\ldots,m_k)$
with $\mathrm{ht}(A)=(i, j),$
then
$A$ is conjugate to a homogeneous generalized matrix unit $u_{i, j}$ in $\frak{s}^{\Phi }(m_1,\ldots,m_k)$ and  $|u_{i,j}|=|A|.$

(2) If $\frak{a}$ is an abelian subalgebra of
$\frak{s}^{\Phi }(m_1,\ldots,m_k)$ with $\mathrm{ht}(\frak{a})=(i , j),$
then there exist some homogeneous generalized matrix units
$u_{i,k_{1}}, u_{i,k_{2}}, \ldots, u_{i,k_{r}}$ in $\frak{s}^{\Phi }(m_1,\ldots,m_k)$,
a homogeneous invertible matrix $T\in\frak{gl}^{\Phi }(m_1,\ldots,m_k)$
and an abelian subalgebra $\frak{a}'$ of
$\frak{s}^{\Phi }(m_1,\ldots,m_k)$
with $\mathrm{ht}(\frak{a}')>(i, m)$,  such that $\frak{a}$ is conjugate to
\begin{equation}\label{GMN decomposition}
T^{-1}\frak{a}T = \mathbb{F}u_{i, k_{1}}\oplus\mathbb{F}u_{i,
k_{2}}\oplus\cdots\oplus\mathbb{F}u_{i, k_{r}}\oplus\frak{a}'
\end{equation}
where $i<j=k_{1}<k_{2}<\cdots<k_{r}\leq m$
and for each matrix $X$ in $\frak{a}'$:
\begin{itemize}
\item
the $k_{j}$th row of $X$ is zero for each $1\leq j\leq r.$
\item
if the $s$th row (resp. column) of $TXT^{-1}$
is zero, so is the $s$th row (resp. column) of $X$.
\end{itemize}
\end{lemma}

\begin{proof}
(1) On one hand, $\mathrm{h}_A(A)$ a generalized matrix unit $u_{i, j}$ since $\mathrm{ht}(A)=(i, j)$.
 On the other hand, since $A$ is homogeneous and $\mathrm{ht}(A)=(i, j),$
$|A|=|e_{ij}|=\Phi_i-\Phi_j$.
Then $\Phi_l=\Phi_j$ if $a_{il}$ is nonzero.
It follows that  $\mathrm{h}_A$ is
in $\mathrm{End}_{0}(\frak{gl}^{\Phi }(m_1,\ldots,m_k))$,
which implies $|u_{i,j}|=|A|$.

(2) Let the homogeneous matrix $A_{ij}\in \frak{a}$ satisfy $\mathrm{ht}\left(A_{ij}\right)=\mathrm{ht}(\frak{a})=(i,j)$.
Then $\mathrm{h}_{A_{ij}}\left(A_{ij}\right)$ is a  generalized matrix unit $u_{i, j}$ by (1),
and $\mathrm{h}_{A_{ij}}\left(\frak{a}\right)$ is a new abelian subalgebra contained in $\frak{s}^{\Phi }(m_1,\ldots,m_k).$
Let $\frak{a}_{(1)}$
be the subspace of $\mathrm{h}_{A_{ij}}\left(\frak{a}\right)$ such that
$\mathrm{ht}(\frak{a}_{(1)})>(i, k_{1})$ and
 any  matrix in $\mathrm{h}_{A_{ij}}\left(\frak{a}\right)$ is of the form
$a u_{i,k_{1}} +P$, where $a\in \mathbb{F}, k_1=j$
and $P\in \frak{a}_{(1)}$.
Similarly, if $\mathrm{ht}(\frak{a}_{(1)})=(i, k_{2})$, let the homogeneous matrix $A_{ik_2}\in \frak{a}_{(1)}$ satisfy
$\mathrm{ht}\left(A_{ik_2}\right)=(i,k_2)$.
Consequently, $\mathrm{h}_{A_{ik_2}}\left(A_{ik_{2}}\right)$
is also a homogenous generalized matrix unit, denoted by $u_{i,k_{2}}$,  and
$\mathrm{h}_{A_{ik_2}}\mathrm{h}_{A_{ik_1}}\left(\frak{a}\right)\supset\mathrm{h}_{A_{ik_2}}\left(\frak{a}_{(1)}\right)$
are also abelian subalgebras of
$\frak{s}^{\Phi }(m_1,\ldots,m_k)$.
In particular, since $k_2>k_1>i$,
$\mathrm{h}_{A_{ik_2}}\left(u_{i,k_{1}}\right)$ is also a generalized matrix unit of $(i,k_{1})$-form,
denoted still by $u_{i,k_{1}}$.
Let $\frak{a}_{(2)}$ be the subspace of $\mathrm{h}_{A_{ik_2}}\left(\frak{a}_{(1)}\right)$  such that
$\mathrm{ht}(\frak{a}_{(2)})>(i, k_{2})$ and
 any a matrix in $\frak{a}$
is of the form: $a_{1}u_{i,k_{1}} +a_{2}u_{i,k_{2} }+P$,
where $ a_{1}, a_{2}\in \mathbb{F}$ and $P\in\frak{a}_{(2)}.$
By induction, there exists a positive integer $r$ such that
$$\mathrm{ht}(\frak{a}_{(r)})>(i, m), \quad 1\leq r\leq m-i$$ and
$$\mathrm{h}_{A_{ik_r}}\mathrm{h}_{A_{ik_{r-1}}}\cdots\mathrm{h}_{A_{ik_{t+1}}}\left(\frak{a}_{(t)}\right)
\subset\mathrm{h}_{A_{ik_r}}\mathrm{h}_{A_{ik_{r-1}}}\cdots\mathrm{h}_{A_{ik_{t}}}\left(\frak{a}_{(t-1)}\right),$$
where $1\leq t\leq r-1$ and $\frak{a}_0=\frak{a}$.
In addition, we also get $r$ homogeneous generalized matrix units:
$u_{i,k_{1}}, u_{i,k_{2}}, \ldots, u_{i,k_{r}}$,
where $i<k_{1}<k_{2}<\cdots<k_{r}\leq m$.
Write
$$\frak{a}'= \frak{a}_{(r)}, \quad
T^{-1}\frak{a}T =\mathrm{h}_{A_{ik_r}}\mathrm{h}_{A_{ik_{r-1}}}\cdots\mathrm{h}_{A_{ik_{1}}}\left(\frak{a}\right).$$
Then
$$T^{-1}\frak{a}T=\mathbb{F}u_{i,k_{1}} \oplus\cdots\oplus\mathbb{F}u_{i,k_{r}} \oplus\frak{a}'.$$
For every $1\leq l\leq r$ and  any  homogeneous matrix $X\in \frak{a}'$,
the $i$th row of $u_{i,k_{l}}X$ is the $k_{l}$th row of $X$
and the $i$th row of $Xu_{i,k_{l}}$ is $0$ by $\mathrm{ht}\frak{a}'>(i, m)$.
Then since $T^{-1}\frak{a}T$ is abelian and every $u_{i,k_{j}}$ is homogeneous,
 the $k_{l}$th row of any   homogeneous matrix in $\frak{a}'$ is $0$.
Furthermore, the operators $\mathrm{h}_{A_{ik_t}}$ leave $0$ rows but the $k_{l}$th rows of any a matrix invariant,
where $1\leq t\leq r$. Then (2) is true.
\end{proof}
(\ref{GMN decomposition})  is called the
\textit{homogeneous generalized matrix unit decomposition} of $\frak{a}$ (HGMU decomposition in
short).
\subsection{Main results}
Write
\begin{equation}\label{2003221402}
\begin{gathered}
\frak{E}=\mathrm{Span}\{e_{ij}\in \frak{gl}^{\Phi }(m_1,\ldots,m_k)\mid
1\leq i\leq \lceil m/2\rceil, \lceil m/2\rceil+1\leq j\leq m\}, \\
\frak{F}=\mathrm{Span}\{e_{ij}\in \frak{gl}^{\Phi }(m_1,\ldots,m_k)\mid
1\leq i\leq \lfloor m/2\rfloor, \lfloor m/2\rfloor+1\leq j\leq m\}, \\
\frak{E'}=\frak{E}\oplus \mathbb{F}\mathrm{I}_{m},\quad\frak{F'}=\frak{F}\oplus \mathbb{F}\mathrm{I}_{m},
\end{gathered}
\end{equation}
where $m=\sum^k_{i=1}m_i$.

Write $\mathbb{Z}_k=\{\bar{0}, \bar{1}, \ldots, \overline{k-1}\}$
and fix the order $\bar{0}, \bar{1}, \ldots, \overline{k-1}$ in the following.
Now we are in the position to determine the pre-nil maximal abelian subalgebras  of $\frak{gl}(m_1,\ldots, m_k)$,
the idea of which mainly comes from Jacobson's paper \cite{Jacobson}.
\begin{theorem}\label{200401947}
Let $\frak{a}$ be a pre-nil maximal  abelian subalgebra of $\frak{gl}(m_1,\ldots, m_k)$ and $m=\sum^k_{i=1}m_i$.
 Suppose $m>3$. Then

(1)
$\frak{a}$ is conjugate to $\frak{E'}$ or $\frak{F'}$
for some $(m_1,\ldots,m_k)$-tuple $\Phi $.

(2) In case $\Gamma=\mathbb{Z}_k$,
$\frak{a}$ is of $\mathbb{Z}_k$-dimension
$$\left( \sum^k_{i=1}\dot{m}_{i+1}\ddot{m}_i,\sum^k_{i=1}\dot{m}_{i+2}\ddot{m}_i,\ldots,\sum^k_{i=1}\dot{m}_{i+k-1}\ddot{m}_i,
\left(\sum^k_{i=1}\dot{m}_{i+k}\ddot{m}_i\right)+1\right),$$
where $\dot{m}_i, \ddot{m}_i$ are nonnegative integers such that $\dot{m}_i+\ddot{m}_i=m_i$
and $\dot{m}_{k+i}=\dot{m}_i$ for $1\leq i \leq k$.
In particular, $\dim\frak{a}=\lfloor m^{2}/4\rfloor+1$.
\end{theorem}
\begin{proof}
(1)
Let $V$ be the natural  $\frak{a}$-module. Similar to the case of Lie algebra, we get
$V=\bigoplus_{\alpha\in \Delta}V_{\alpha},$
where $$V_{\alpha}=\{v\in V\mid \mbox{for every $ x\in \frak{a},$ there exists $l\in \mathbb{N}$
such that $(x-\alpha(x)\mathrm{id}_{V})^{l}v=0$}\}.$$
We use induction on $|\Delta|$. 

(a) $\mathbf{|\Delta|=1:}$
From Lemma \ref{1219},   there exists  an $(m_1,\ldots,m_k)$-tuple $\Phi$
such that $\frak{a}$ is contained in $\frak{t}^{\Phi}(m_1,\ldots,m_k)$.
Let $\frak{b}$ be the abelian subalgebra in
$\frak{s}^{\Phi }(m_1,\ldots,m_k)$
such that $\frak{a}=\mathbb{F}\mathrm{I}_{m}\oplus \frak{b}$.
From Lemma \ref{1527}, we  get the following HGMU decompositions:
\begin{eqnarray}\label{2231}
\begin{array}{cccc}
T^{-1}_{1}\frak{b}T_{1} =
\mathbb{F}u_{i_{1}, j_{11}}\oplus\cdots\oplus\mathbb{F}u_{i_{1},
j_{1r_{1}}}\oplus\frak{b}_{1}\\
T^{-1}_{2}\frak{b}_{1}T_{2} = \mathbb{F}u_{i_{2},
j_{21}}\oplus\cdots\oplus\mathbb{F}u_{i_{2},
j_{2r_{2}}}\oplus\frak{b}_{2}\\
\cdots\;\cdots\\
T^{-1}_{t}\frak{b}_{t-1}T_{t} = \mathbb{F}u_{i_{t},
j_{t1}}\oplus\cdots\oplus\mathbb{F}u_{i_{t}, j_{t
r_{t}}}\oplus\frak{b}_{t}
\end{array}
\end{eqnarray}
where
\begin{itemize}
\item
$(i_{l+1}, j_{l+1,1})=\mathrm{ht}(\frak{b}_{l})>(i_l, m)$
for $0\leq l \leq t-1$
with $\frak{b}_{0}=\frak{b}, \frak{b}_t=0$, and then $\dim \frak{b}=r_{1}+r_{2}+\cdots+r_{t}$;
\item
generalized matrix units in (\ref{2231})
are homogeneous, whose index-pairs
satisfy
\begin{eqnarray}\label{wenzi}
\mbox{no indexes in the first position appear in the second position.}
\end{eqnarray}
\end{itemize}
For $1\leq l\leq t$, the fact $i_l>i_{l-1}>\cdots>i_1$ and (\ref{wenzi}) imply $r_l\leq m-i_l-(t-l)$,
and then
\begin{equation}\label{2004052128}
\begin{split}
\dim\frak{b}&=r_{1}+\cdots+r_{t}\\
&\leq tm-(i_{1}+\dots+i_{t})-(1+\cdots+t-1)\\
&\leq tm-(1+\cdots+t)-(1+\cdots+t-1)\\
&=t(m-t)\\
&\leq\lfloor m^{2}/4\rfloor.
\end{split}
\end{equation}
Since
$\dim \frak{a}\geq \dim \frak{E'}=\lfloor m^{2}/4\rfloor+1$,
 $\dim \frak{b}\geq\lfloor m^{2}/4\rfloor.$
This forces (\ref{2004052128}) is a equation, that is,
\begin{equation}\label{eq1514t}
\mbox{$i_{l}=l,\quad r_{l}=m-t, \quad t=\lceil m/2\rceil$ or $\lfloor m/2\rfloor$},
\end{equation}
for $1\leq l\leq t$.

 If $m=2n$ is even, then
$$\mbox{$i_l=l,\quad t =n,\quad r_{l}=m-t=n$
for  $1\leq l\leq n$}$$
 by (\ref{eq1514t}).
Then by (\ref{wenzi}),  generalized matrix units in (\ref{2231})
are $u_{i,j}$ for $1\leq i\leq n, n+1\leq j\leq 2n$.
It follows that the last $n$ rows of any matrix in
$\frak{b}_1\cup\frak{b}_2\cdots\cup\frak{b}_{n-1}$ is zero rows
from Lemma \ref{1527}(2).
In particular, $u_{n,j}$ in (\ref{2231}) is $e_{nj}$, where $n+1\leq j\leq 2n$.
Furthermore,  $u_{n-1,j}$ in (\ref{2231}) may be viewed as  $e_{n-1j}$, $n+1\leq j\leq 2n$.
By induction, $u_{l,j}$ in (\ref{2231}) may be viewed as $e_{lj}$, where $n+1\leq j\leq 2n$ and $2\leq l\leq n$.
Then $\frak{b}_1$ is spanned by $\left\{e_{ij}\mid 2\leq i\leq n, n+1\leq j\leq 2n\right\}$.
Since $$0=[u_{l,h}, e_{ij}]=u_{l,h}e_{ij}-\varepsilon(|u_{l,h}|, |e_{ij}|)e_{ij}u_{l,h}=-\varepsilon(|u_{l,h}|, |e_{ij}|)e_{ij}u_{l,h}$$
for $1\leq i\leq n, n+1\leq j, h\leq 2n$,
the last $n$ rows
of each $u_{l,h}$ are zero rows.
Consequently,  $\frak{b}$ is conjugate to $\frak{E}$ or $\frak{F}$,
and $\frak{a}=\mathbb{F}\mathrm{I}_{m}\oplus \frak{b}$ is conjugate to
 $\frak{E'}$ or $\frak{F'}$.

The remaining case $s=2n+1$ can be analogously treated.

(b) $\mathbf{|\Delta|>1:}$
Let $\Delta=\{\beta_{1}, \beta_{2}, \ldots, \beta_{l}\}$ where $l>1.$
We may suppose that
$\frak{a}=\frak{a}_{1}\oplus\frak{a}_{2}$,
where $\frak{a}_{i}$ is  an abelian subalgebra of $\frak{gl}(V_{i})$
with $V_{1}=V_{\beta_{1}}$ and
$V_{2}= \bigoplus^{l}_{i=2}V_{\beta_{i}}$.
For $i=1, 2$, denote by $\Delta_{i}$ the weight set  of $V_{i}$ with respect to $\frak{a}_{i}$.
Note that
 $\Delta_{1}=\{\beta_{1}\}$ and $\Delta_{2}=\{\beta_{j}\mid 2\leq j\leq l\}$
 by abuse of language.
Since $|\Delta_{i}|<|\Delta|$ for $i=1, 2$, by inductive hypothesis we may assume that the theorem holds for
$\frak{gl}(V_{1})$ and $\frak{gl}(V_{2}).$
Let $(\dot{m}_{1},\ldots,\dot{m}_k)$ or
$(\ddot{m}_{1},\ldots,\ddot{m}_{k})$ be $\Gamma$-dimension of $V_1$ or $V_2$, respectively.
Write $\dot{m}=\sum^k_{i=1}\dot{m}_{i}$ and $\ddot{m}=\sum^k_{i=1}\ddot{m}_{i}$.
Then $m=\dot{m}+\ddot{m}$ and $m_i=\dot{m}_i+\ddot{m}_i$, $1\leq i\leq k$.

(\MyRoman{1})
$m=2t-1, \dot{m}=2t_{1}-1, \ddot{m}=2t_{2}$:
Here $t=t_{1}+t_{2}$ and
$$\dim\frak{a}\leq t_{1}(t_{1}-1)+1+t_{2}^{2}+1\leq t(t-1)+1.$$
Equality holds between the last terms only when $m=3.$

(\MyRoman{2})
$m=2t, \dot{m}=2t_{1}-1, \ddot{m}=2t_{2}-1$:
Here $t=t_{1}+t_{2}-1$ and
$$\dim\frak{a}\leq t_{1}(t_{1}-1)+1+t_{2}(t_{2}-1)+1\leq t^{2}+1.$$
Equality holds between the last terms only when $m=2.$

(\MyRoman{3})
$m=2t, \dot{m}=2t_{1}, \ddot{m}=2t_{2}$:
Here $t=t_{1}+t_{2}$ and
$$\dim\frak{a}\leq t_{1}^{2}+1+t_{2}^{2}+1<t^{2}+1.$$

Hence (1) holds.

(2)
From (1), it is sufficient to consider the cases
$\frak{a}=\frak{E'}$ and $\frak{F'}$.
Let $\dot{m}_i$ be the cardinality of the set
$$\mbox{$\left\{j\mid \Phi_j=\bar{i}, 1\leq j\leq \lfloor\frac{m}{2}\rfloor\right\}\quad$ if $\quad\frak{a}=\frak{F'}$}$$
or
$$\mbox{$\left\{j\mid \Phi_j=\bar{i}, 1\leq j\leq \lceil\frac{m}{2}\rceil\right\}\quad$ if $\quad\frak{a}=\frak{E'}$,}$$
and write $\ddot{m}_i=m_i-\dot{m}_i$,
where $1\leq i\leq k$.
For $1\leq l\leq k$, denote by $\frak{a}_{\bar{l}}$ the homogeneous subspace of degree $\bar{l}$ for $\frak{a}$.
Since the matrix unit $e_{st}$ is of the degree $\Phi_s-\Phi_t$, we get
$$\frak{a}_{\bar{l}}=\mathrm{Span}\left\{e_{st}\mid \Phi_s=\overline{i+l}, \Phi_t=\bar{i}, 1\leq i\leq k\right\}+\delta_{\bar{l},0}I_m.$$
Hereafter $\delta_{i,j}$ is 1 if $i=j$ and 0 otherwise.
It follows that $$\dim \frak{a}_{\bar{l}}=\sum^k_{i=1}\dot{m}_{i+l}\ddot{m}_i+\delta_{\bar{l},0}1.$$
\end{proof}

By a direct computation, we get the following  theorem, which complements the Theorem \ref{200401947}.
\begin{theorem}\label{200406055}
Let $\frak{a}$ be a pre-nil maximal  abelian subalgebra of $\frak{gl}(m_1,\ldots, m_k)$ and $m=\sum^k_{i=1}m_i$.
\begin{itemize}
\item[(1)]
If $m=3$, $\frak{a}$ is conjugate to one of the following
$$\mathrm{Span}\left\{\mathrm{I}_3, e_{12}, e_{13}\right\}, \quad\mathrm{Span}\left\{\mathrm{I}_3, e_{23}, e_{13}\right\},$$
$$\mathrm{Span}\left\{e_{11}+e_{22}, e_{12}, e_{33}\right\}, \quad\mathrm{Span}\left\{e_{11}, e_{22}, e_{33}\right\}.$$
\item[(2)]
If $m=2$, $\frak{a}$ is conjugate to
$\mathrm{Span}\left\{\mathrm{I}_2, e_{12}\right\}$ or $\mathrm{Span}\left\{e_{11}, e_{22}\right\}.$
\end{itemize}
\end{theorem}
Since any abelian subalgebra of the general linear Lie superalgebra is a pre-nil one,
Theorems \ref{200401947} and  \ref{200406055} cover Schur's work and its super-version.
As a by-product,
we may determine all nil maximal abelian subalgebras of $\frak{gl}(m_1,\ldots, m_k)$  as follows.
\begin{corollary}\label{2004052358}
Let $\frak{b}$ be a nil maximal abelian subalgebra of $\frak{gl}(m_1,\ldots, m_k)$ and $m=\sum^k_{i=1}m_i$.

 (1) Suppose $m>3$. Then
\begin{itemize}
\item
$\frak{b}$ is conjugate to $\frak{E}$ or $\frak{F}$
for some $(m_1,\ldots,m_k)$-tuple $\Phi $.
\item
If $\Gamma=\mathbb{Z}_k$, then $\frak{b}$ is of $\Gamma$-dimension
$$\left(\sum^k_{i=1}\dot{m}_{i+1}\ddot{m}_i,\sum^k_{i=1}\dot{m}_{i+2}\ddot{m}_i,\ldots,\sum^k_{i=1}\dot{m}_{i+k-1}\ddot{m}_i,
\sum^k_{i=1}\dot{m}_{i+k}\ddot{m}_i\right),$$
where $\dot{m}_i, \ddot{m}_i$ are nonnegative numbers such that $\dot{m}_i+\ddot{m}_i=m_i$,
and $\dot{m}_{k+i}=\dot{m}_i$ for $1\leq i \leq k$.
In particular, $\dim \frak{b}=\lfloor m^{2}/4\rfloor$.
\end{itemize}

(2) If $m=3$, then $\frak{b}$  is conjugate to one of the following
$$\mathrm{Span}\left\{e_{12}, e_{13}\right\},
\quad\mathrm{Span}\left\{e_{23}, e_{13}\right\},
\quad\mathrm{Span}\left\{e_{12}, e_{33}\right\}.$$

(3) If $m=2$, then $\frak{b}$ is conjugate to
$\mathrm{Span}\left\{e_{12}\right\}$.
\end{corollary}
From Theorems \ref{200406055} and \ref{2004052128}, we have the following remark.
If we focus only on the maximal dimension of pre-nil  or nil abelian subalgebras for $\frak{gl}(m_1, \ldots, m_k)$,
we may give a direct proof of  the following remark by virtue of Mirzakhani's idea in \cite{Mirzakhani},
for which readers  may see the Appendix.
\begin{remark}\label{200406129}
(1) Any nil maximal abelian subalgebra of $\frak{gl}(m_1, \ldots, m_k)$ is of dimension
 $\lfloor\frac{\left(\sum^k_{i=1}m_i\right)^2}{4}\rfloor$.

 (2) Any pre-nil maximal  abelian subalgebra of $\frak{gl}(m_1, \ldots, m_k)$ is of dimension
  $\lfloor\frac{\left(\sum^k_{i=1}m_i\right)^2}{4}\rfloor+1$.

\end{remark}

\subsection{Applications}
For any $\mathbb{Z}_k$-graded  abelian Lie color algebra $L$,
the following theorem gives  a method of determining $\mu_{\text{pre-nil}}(L)$  and $\mu_{\text{nil}}(L)$.
\begin{theorem}\label{200406126}
Let $L$ be an abelian Lie color algebra.
\begin{itemize}
\item[(1)]
$\mu_{\text{pre-nil}}(L)\geq \lceil2\sqrt{\dim L-1}\rceil,\quad \mu_{\text{nil}}(L)\geq \lceil2\sqrt{\dim L}\rceil.$
\item[(2)]
In case $\Gamma=\mathbb{Z}_k$,
$L$ possesses a pre-nil faithful representation of dimension $m$ if and only if
$m$ admits a $2k$-partition $(\dot{m}_1, \ddot{m}_1, \ldots, \dot{m}_k, \ddot{m}_k)$ such that
$\sum^k_{i=1}\dot{m}_{i+l}\ddot{m}_i+\delta_{\bar{l},0}1\geq\dim L_{\bar{l}}$,
where $\dot{m}_{k+i}=\dot{m}_i$ and $1\leq l\leq k-1$.
\item[(3)]
In case $\Gamma=\mathbb{Z}_k$,
$L$ possesses a nil faithful representation of dimension $m$ if and only if
$m$ admits a $2k$-partition $(\dot{m}_1, \ddot{m}_1, \ldots, \dot{m}_k, \ddot{m}_k)$ such that
$\sum^k_{i=1}\dot{m}_{i+l}\ddot{m}_i\geq\dim L_{\bar{l}}$,
where $\dot{m}_{k+i}=\dot{m}_i$ and $1\leq l\leq k-1$.
\end{itemize}
\end{theorem}
\begin{proof}
(1)
Let
$\iota: L\longrightarrow \frak{gl}(m_1, \ldots, m_k)$ be  a pre-nil  faithful representation of $L$.
Then  $\iota(L)$ is a pre-nil  abelian subalgebra of $\frak{gl}(m_1, \ldots, m_k)$.
From Theorem \ref{200401947}(1),
$$\dim L=\dim \iota(L)\leq \lfloor (\sum^k_{i=1}m_i)^{2}/4\rfloor+1.$$
Consequently,  $\sum^k_{i=1}m_i\geq \lceil2\sqrt{\dim L-1}\rceil$.

Furthermore, if $\iota$ is nil, then each element in $\iota(L)$ is nilpotent.
From Corollary \ref{2004052358},
$$\dim L=\dim \iota(L)\leq \lfloor (\sum^k_{i=1}m_i)^{2}/4\rfloor.$$
Consequently,  $\sum^k_{i=1}m_i\geq \lceil2\sqrt{\dim L}\rceil$.

(2) and (3) are true by virtue of Theorem \ref{2004052128}(2) and Corollary \ref{2004052358}(2), respectively.
\end{proof}

\section{Appendix: The maximal dimension  of pre-nil  or nil abelian subalgebras for $\frak{gl}(m_1,\ldots,m_k)$}
In this section, we give a direct proof of Remark \ref{200406129}. The main idea comes from
 Mirzakhani's work \cite{Mirzakhani}.

Let $\frak{a}$ be any maximal pre-nil abelian subalgebra of $\frak{gl}(m_1,\ldots,m_k)$.
 We may assume that
$\frak{a}$ is contained in
$\frak{t}^{\Phi }(m_1,\ldots,m_k)$ from  Lemma  \ref{1219}.
Let us use induction on $m$ to show that
$\dim \frak{a}\leq\lfloor m^{2}/4\rfloor+1$.
When $m=1$,  the conclusion holds
since $\frak{gl}^{\Phi }(m_1,\ldots,m_k)$ is of dimension 1.
Assume that the conclusion holds for $m-1$.
Let us consider the case  $m$.
Suppose the contrary, that is,
$\dim \frak{a}>\lfloor m^{2}/4\rfloor+1$.
Then $\frak{a}$ contains an abelian subalgebra, say $\frak{n}$, of dimension
\begin{equation}\label{1251}
\nu(m):=\lfloor m^{2}/4\rfloor +2.
\end{equation}
Fix   a homogeneous  basis of $\frak{n}$:
$\{A_{i}\mid 1\leq i\leq\nu(m)\}.$
For $1\leq i\leq \nu(m)$, let $\bar{A}_{i}$  and $\tilde{A}_{i}$
be $(m-1)\times (m-1)$ matrices such that
$$A_i=\small{\left[
\begin{array}{cccc} a_{11}&\ &\cdots&a_{1m}\\
\ &\ &\ &\ \\\vdots&\ &\bar{A}_i&\
\\a_{m1}&\ &\ &\
 \end{array}\right]}=\left[
\begin{array}{ccccc}
 \ &\ &\ &a_{1m}\\ \ &\tilde{A}_i&\ &\vdots\\\ &\ &\ &\ \\ a_{m1}&\cdots &\ &a_{mm}\end{array}\right].$$
Write
$$
\mbox{$\bar{\Phi }=(\hat{\Phi }_{1}, \Phi _{2}, \ldots, \Phi _{m})$
and $\tilde{\Phi }=(\Phi _{1}, \Phi _{2}, \ldots, \hat{\Phi }_{m})$,}
$$
where  the sign $\hat{\;}$ means that the element under it is omitted.
It is obvious that $\bar{\Phi }$
is  an $(m_1, \ldots, m_g-1, \ldots, m_k)$-tuple if $\Phi _1=\alpha_g$,
 $\tilde{\Phi }$ is  an $(m_1, \ldots, m_h-1, \ldots, m_k)$-tuple if $\Phi _m=\alpha_h$.
Then $\bar{A}_{i}$  and $\tilde{A}_{i}$ are in $\frak{gl}^{\bar{\Phi }}(m_1, \ldots, m_g-1, \ldots, m_k)$
and $\frak{gl}^{\tilde{\Phi }}(m_1, \ldots, m_h-1, \ldots, m_k)$ respectively, and $|\bar{A}_i|=|A_i|=|\tilde{A}_i|$.
By $[A_i, A_j]=0$, we have $A_iA_j=\varepsilon\left(|A_i|, |A_j|\right)A_jA_i$. Then
\begin{eqnarray*}
&&\small{\left[
\begin{array}{cccc} \ast&\ &\cdots&\ast\\0&\ &\ &\ \\\vdots&\ &\bar{A_j}\bar{A_j}&\
\\0&\ &\ &\
 \end{array}\right]}=\varepsilon\left(|A_i|, |A_j|\right)\small{\left[
\begin{array}{cccc} \ast&\ &\cdots&\ast\\0&\ &\ &\ \\\vdots&\ &\bar{A_j}\bar{A_i}&\
\\0&\ &\ &\
 \end{array}\right],}\\
&&\left[
\begin{array}{ccccc}
 \ &\ &\ &\ast\\ \ &\tilde{A_i}\tilde{A_j}&\ &\vdots\\\ &\ &\ &\ \\ 0&\cdots &0&\ast\end{array}\right]=
 \varepsilon\left(|A_i|, |A_j|\right)\left[
\begin{array}{ccccc}
 \ &\ &\ &\ast\\ \ &\tilde{A_j}\tilde{A_i}&\ &\vdots\\\ &\ &\ &\ \\ 0&\cdots &0&\ast\end{array}\right].
\end{eqnarray*}
That is, $[\bar{A_i}, \bar{A_j}]=0=[\tilde{A_i}, \tilde{A_j}]$.
Let $\bar{W}$ and $\tilde{W}$ be the $\Gamma$-graded vector spaces spanned by
$\{\bar{A}_{i}\mid 1\leq i\leq\nu(m)\}$ and $\{\tilde{A}_{i}\mid 1\leq i\leq\nu(m)\}$, respectively.
Then $\bar{W}$ and $\tilde{W}$
are also abelian subalgebras of $\frak{t}^{\bar{\Phi }}(m_1,\ldots,m_g-1,\ldots,m_k)$ and
$\frak{t}^{\tilde{\Phi }}(m_1,\ldots,m_h-1,\ldots,m_k)$ for some $g$ and $h$, respectively.
  Write $r=\dim \bar{W}$ and $t=\dim \tilde{W}$.
 By inductive hypothesis, we have
\begin{equation} \label{1239}
\begin{gathered}
r\leq\lfloor(m-1)^{2}/4\rfloor+1,\\
t\leq\lfloor(m-1)^{2}/4\rfloor+1.
\end{gathered}
\end{equation}
 Without loss of generality, we may assume  that
$\{\bar{A}_{i}\mid 1\leq i\leq r\}$
are linearly independent, so are $\{\tilde{A}_{i}\mid 1\leq i\leq t\}$.
Let
\begin{equation*}\label{equation}
\mbox{$\bar{A}_{i}=\sum^{r}_{k=1}\bar{m}_{ik}\bar{A}_{k},$ where
$\bar{m}_{ik}\in \mathbb{F}$ and $ r+1\leq i\leq\nu(m)$,}
\end{equation*}
\begin{equation*}\label{equation1}
\mbox{$\tilde{A}_{j}=\sum^{t}_{k=1}\tilde{m}_{jk}\tilde{A}_{k},$ where
$\tilde{m}_{jk}\in \mathbb{F}$ and  $ t+1\leq j\leq\nu(m)$}.
\end{equation*}
Note that $|\bar{A}_{i}|=|\bar{A}_{j}|$
if  $\bar{m}_{ij}\neq 0$,
and $|\tilde{A}_{i}|=|\tilde{A}_{j}|$
if  $\tilde{m}_{ij}\neq 0$.
Then
$$\mbox{ $|A_{i}|=|A_{j}|$
if  $\bar{m}_{ij}\neq 0$ or  $\tilde{m}_{ij}\neq 0$.}$$
For $r+1\leq i\leq\nu(m)$ and $t+1\leq j\leq\nu(m)$,
write
$$\bar{B}_{i}=A_{i}-\sum^{r}_{k=1}\bar{m}_{ik}A_{k}, \quad \tilde{B}_{j}=A_{j}-\sum^{t}_{k=1}\tilde{m}_{jk}A_{k}.$$
Thus each $\bar{B}_{i}$ (resp. $\tilde{B}_{j}$) is  homogeneous  and
of the form $[\bar{b}_{i},\; O]^{\mathrm{t}}$
(resp. $\left[O, \tilde{b}_{j}\right]$),
where $\bar{b}_{i}^{\mathrm{t}}=\bar{a}_{i}-\sum^{r}_{k=1}\bar{m}_{ik}\bar{a}_{k}$ is a
$1\times m$ matrix and  $\bar{a}_{q}$ is the first row of $A_{q}$ for
$1\leq q\leq\nu(m)$
(resp. $\tilde{b}_{j}=\tilde{a}_{j}-\sum^{l}_{k=1}\tilde{m}_{jk}\tilde{a}_{k}$   is an $m\times 1$
matrix and $\tilde{a}_{q}$ is the last column of $A_{q}$, $1\leq q\leq \nu(m)$).
Hereafter $X^{\mathrm{t}}$ denotes the
transpose of a matrix $X$.
Clearly,
$$\mbox{$\left\{\bar{B}_{i}\mid r+1\leq i\leq\nu(m)\right\}$
(resp. $\left\{\tilde{B}_{i}\mid t+1\leq i\leq\nu(m)\right\}$)}$$
 are linearly independent,
and so are
$$\mbox{$\left\{\bar{b}^{\mathrm{t}}_{i}\mid r+1\leq i\leq\nu(m)\right\}$
(resp. $\left\{\tilde{b}_{i}\mid t+1\leq i\leq\nu(m)\right\}$).}$$
Let $M=\left[\bar{b}_{r+1}, \bar{b}_{r+2}, \ldots, \bar{b}_{\nu(m)}\right]^{\mathrm{t}}$.
Clearly,
\begin{equation}\label{1256}
\mathrm{rank} M=\nu(m)-r.
\end{equation}
Denote by  $W$ the set $\left\{X\in \mathbb{F}^{m}\mid MX=0\right\}$.
Then
\begin{equation}\label{1254}
\dim W = m-\mathrm{rank} M.
\end{equation}
For $r+1\leq i\leq\nu(m)$ and $ t+1\leq j\leq\nu(m)$,
 $\bar{B}_{i}, \tilde{B}_{j}$ are homogenous matrices in $\frak{n}$.
Note that $\tilde{B}_{j}\bar{B}_{i}=0$, then $\bar{B}_{i}\tilde{B}_{j}=0$ by $[\bar{B}_{i}, \tilde{B}_{j}]=0$.
Consequently,
$$\mbox{$\bar{b}^{\mathrm{t}}_{i}\tilde{b}_{j}=0$
for $r+1\leq i\leq\nu(m)$ and $ t+1\leq j\leq\nu(m)$},$$
that is, the set $\left\{\tilde{b}_{j}\mid t+1\leq j\leq\nu(m)\right\}$ is contained in $W$, which are  linearly independent.
Consequently,
\begin{equation}\label{1257}
\dim W\geq\nu(m)-t.
\end{equation}
Therefore
\begin{eqnarray*}
m&\stackrel{(\ref{1254})}{=}&\mathrm{rank} M+\dim W\\
&\stackrel{(\ref{1256})(\ref{1257})}{\geq}&\nu(m)-r+\nu(m)-t\\
&\stackrel{(\ref{1251})}{=}&2(\lfloor m^{2}/4\rfloor +2)-r-t\\
&\stackrel{(\ref{1239}) }{\geq}&2(\lfloor
m^{2}/4\rfloor-\lfloor(m-1)^{2}/4\rfloor+1).
\end{eqnarray*}
Thus, if $m=2q$ is even, then
$2q\geq2(q+1)$, a contradiction; if $m=2q+1$ is odd, then $2q+1\geq2(q+1)$,
also a contradiction.
Hence $\dim \frak{a}\leq\lfloor m^{2}/4\rfloor+1$.
By (\ref{2003221402}), we have  $\dim \frak{a}\geq=\dim \frak{E}'=\lfloor m^{2}/4\rfloor+1.$

Furthermore, if each element in $\frak{a}$ is nilpotemt, we may assume that
$\frak{a}$ is contained in $\frak{s}^{\Phi }(m_1,\ldots,m_k)$ from  Lemma  \ref{1219}.
Thus $$\lfloor m^{2}/4\rfloor=\dim\frak{E}\leq\dim \frak{a}<\lfloor m^{2}/4\rfloor+1.$$

\textbf{ACKNOWLEDGMENTS}

This study was supported by the NSF of Heilongjiang
Province (YQ2020A005) and the NSF of China (12061029). 
There are no relevant financial or non-financial competing interests in this paper.  
No potential competing interest was reported by the authors.


\end{document}